\documentclass[final,12pt]{amsart}

\usepackage{a4}
\usepackage[english]{babel}                

\usepackage[T1]{fontenc}
\usepackage[ansinew]{inputenc}
\usepackage{amsmath}                       
\usepackage{amssymb}                       
\usepackage{amsfonts}                      
\usepackage{amsthm}     
\usepackage{enumerate}                     

\newcommand{\RR}{\mathbb R}
\newcommand{\NN}{\mathbb N}

\newcommand{\cU}{\mathcal{U}}
\newcommand{\eps}{\varepsilon}

\newcommand{\abs}[1]{\left\vert#1\right\vert}

\newcommand{\nabs}[1]{\vert#1\vert}

\newcommand{\measN}[1]{\abs{#1}}

\newcommand{\norm}[1]{\left\|#1\right\|}

\newcommand{\nnorm}[1]{\|#1\|}

\newcommand{\mysetl}[2] {\left\{\left.#1\,\right|\,#2\right\}}

\newcommand{\esssup}{\operatorname{ess\: sup}}
\newcommand{\loc}{\text{loc}}
\newcommand{\rad}{\text{rad}}

\renewcommand{\div}{\operatorname{div}}

\newcommand{\curl}{\operatorname{curl}}

\newcommand{\mbf}[1]{\boldsymbol{#1}}
\newcommand{\To}{\longrightarrow}

\newenvironment{myenum}{\vspace*{-1ex}\begin{enumerate}}{\end{enumerate}\vspace*{-1ex}}

\newtheorem{theorem}{Theorem}

\newtheorem{thm}[theorem]{Theorem}
\newtheorem{cor}[theorem]{Corollary}
\newtheorem{lem}[theorem]{Lemma}
\newtheorem{prop}[theorem]{Proposition}
\theoremstyle{definition}
\newtheorem{defn}[theorem]{Definition}
\newtheorem{ex}[theorem]{Example}
\theoremstyle{remark}
\newtheorem{rem}[theorem]{Remark}
\renewenvironment{proof}[1][\proofname]{\par
  \normalfont
  \trivlist
  \item[\hskip\labelsep\itshape
    \bfseries{#1.}]\ignorespaces
}{%
  \qed\endtrivlist
}

\newcommand{\R}{{\mathbb R}}

\newcommand{\N}{{\mathbb N}}

\title[Weak-strong convergence and applications]{A weak-strong convergence property and symmetry of minimizers of constrained variational problems in $\mathbb{R}^N$}

\author[H.\ Hajaiej]{Hichem Hajaiej}
\author[S.\ Kr\"omer]{Stefan Kr\"omer}

\address{Ipeit (Institut preparatoire aux etudes d'ingenieur de Tunis)
\newline\indent
 2, Rue Jawaher Lel Nahru -
1089 Montfleury - tunis
 Tunisie}
\email{hichem.hajaiej@gmail.com}

\address{Mathematisches Institut, Universit\"at zu K\"oln
\newline\indent
50923 K\"oln, Germany}
\email{skroemer@math.uni-koeln.de}

\thanks{ HH
was partially supported by the Tunisian ARUB project : {\em Analyse
Math\'ematique et Applications 04/UR/15-02. }}

\begin{document}

\begin{abstract}
We prove a weak-strong convergence result for functionals of the form $\int_{\mathbb{R}^N} j(x, u, Du)\,dx$ on $W^{1,p}$, 
along equiintegrable sequences. We will then use it to study cases of equality in the extended Polya-Szeg\"o inequality and discuss applications of such a result to prove the symmetry of minimizers of a class of variational problems including nonlocal terms under multiple constraints. 
\end{abstract}

\maketitle

\section{Introduction}

Weak-strong convergence results have attracted many mathematicians during the last decades. 
In the simplest case, we search for hypotheses on an integrand $A$ such that 
\begin{equation}
\label{WS1}
\left.
\begin{array}{l}
u_n \rightharpoonup u \text{ in } L^p(\Omega;\RR^m),\\
\int_\Omega A(u_n)\,dx \rightarrow \int_\Omega A(u)\,dx
\end{array}
\right\}
~\Longrightarrow~u_n\rightarrow u~\text{(strongly)}
\end{equation}
On a bounded domain $\Omega$, appropriate conditions 
were determined by Visintin~\cite{Vi84a}, with strict convexity of $A$ playing a crucial role. Of course, these remain sufficient if the sequence $u_n$ is constrained to some subset of $L^p$, for instance in the gradient case where
$u_n=D v_n$ for some $v_n\in W^{1,p}$. For this particular case and more general integrands, 
refinements were obtained in several powerful papers~\cite{Sy98a},~\cite{Sy99a} and~\cite{Zha89a}, 
showing that
\begin{equation}
\label{WS2}
\left.
\begin{array}{l}
u_n \rightharpoonup u \text{ in } W^{1,p}(\Omega;\RR^m),\\
\int_\Omega j(x, u_n, Du_n)\,dx \rightarrow \int_\Omega j(x,u,Du)\,dx,
\end{array}
\right\}
~\Longrightarrow~u_n\rightarrow u~\text{(strongly)}
\end{equation} 
in particular relying on strict convexity of $j$ in the gradient variable.
However, all of these articles yield strong convergence at most on bounded domains. Results on unbounded domains with the sequence constrained to solutions of a linear system of differential equations 
(which includes the gradient case by using the constraint $\curl u_n=0$, but is not limited to it) were recently obtained by the second author in \cite{Kroe09cp}.

The aim of this paper is to derive a specialized weak-strong convergence result of the form \eqref{WS2} for $\Omega=\RR^N$,
as a crucial tool to study cases 
of equality in a generalized Polya-Szeg\"o inequality. As we shall see below, for this purpose, it suffices to consider sequences $u_n$ that satisfy suitable equiintegrability conditions (namely sequences of iterated polarizations, cf.~Section 2), which allows us to dispose of coercivity assumptions on $j$ that otherwise would be needed
(in fact, without coercivity, for general sequences the best one can hope to obtain from the premises of \eqref{WS2} and strict convexity of $F$ in the gradient variable is strong convergence in $W^{1,r}_\loc$ for $1\leq r<p$). We then present an application of this result to 
prove the symmetry of all the minimizers of a class of functionals involving terms of the form $\int j_i(x, u_i, Du_i)$, a local and a nonlocal nonlinearity, for vector-valued $U=(u_1,\ldots,u_m)$ whose components are constrained to spheres in $L^p$.
Here, due to the constraint, we only need coercivity of $j_i$ in $D u_i$ (and not in $u_i$) to obtain minimizers.
 
Let us first put the reader in the general framework of our study.

Let $m\in\N$ and $U= (u_1, \ldots, u_m)\in \left[W_+^{1,p}(\R^N)\right]^m$, where $1<p<\infty$ and
$W_+^{1,p}(\R^N)$ denotes the cone of non-negative functions belonging to $W^{1,p}(\R^N)$.

Ultimately, we are interested in symmetry properties of minimizers of the variational problem  
\begin{equation}
\label{PC}
\inf\limits_{U\in S} E(U)=:I,
\end{equation}
where 
$E(U):=\textstyle{\sum_{j=1}^k} E^j(U)$ for some $E^j:\left[W^{1,p}_+(\R^N)\right]^m\to \RR$. The functional 
is constrained to a set $S\subset W^{1,p}_+(\R^N)^m$ with the following property: 
$$
	\text{If $U\in S$ then $U^H=\left(u_1^H, \ldots, u_m^H\right)$ and $U^*=\left(u_1^*, \ldots, u_m^*\right)$ are also in $S$}. 
$$
Here, $u_i^H$ denotes the polarization (also called two-point rearrangement, e.g., \cite{BroSo00}) of $u_i$ with respect 
to an closed half-space $H$ containing the origin. The set of such half-spaces is denoted by 
$\mathcal{H}$ below, and $u_i^*$ is the Schwarz rearrangement (or radial nonincreasing reaarangement, e.g., \cite{Ka85B}) of $u_i$.

If for any $1\leq j\leq k$, 
\begin{equation}
	\label{0.0i}
	E^j(U^*)\leq E^j(U^H) \leq E^j(U), \quad\forall U\in W_+^{1,p}(\R^N)^m ~\forall H\in\mathcal{H}
\end{equation}
and thus  
\begin{equation}
	\label{0.0}
	E(U^*)\leq E(U^H) \leq E(U), \quad\forall U\in W_+^{1,p}(\R^N)^m ~\forall H\in\mathcal{H},
\end{equation}
then it is sufficient to consider a Schwarz symmetric minimizing sequence of \eqref{PC},
i.e., each component of the sequence is radial and radially decreasing. Therefore, 
the existence of a minimizer of \eqref{PC} becomes less difficult to prove 
thanks to the compact embedding of~$W_{rad}^{1,p}(\R^N)$ in $L^q(\R^N)$ for some appropriate $q$ (cf.~Lemma~\ref{lem:radcompact} below). 

Moreover, to obtain symmetry properties of the solutions of \eqref{PC}, 
one is lead to study the cases of equality in \eqref{0.0i}. We are thus interested
in finding suitable assumptions under which
\begin{equation}
	\label{0.1i}
	E^j(U^*)= E^j(U) \Rightarrow U=U^*~~\text{(up to a translation)}.
\end{equation}
Of course if \eqref{0.0i} holds true for any $1\leq j\leq k$ and \eqref{0.1i} is true for one~$j$, 
then it can be easily deduced that all the minimizers of \eqref{PC} (if they exist) are 
Schwarz symmetric up to a translation.  

In our application in Section~\ref{sec:app}, we consider a functional given by three summands $E^j$. For $U=(u_1,\ldots,u_m)\in W^{1,p}(\RR^N)^m$, 
the first is an energy term of the form
\begin{equation}
	\label{0.01}
	E^1(U)=\sum_{i=1}^m \int_{\R^N}j_i(x,u_i,\abs{Du_i})\,dx.
\end{equation}
The second term $E^2$ adds a local integral functional of lower order, namely, 
\begin{equation}
	\label{0.02}
	E^2(U)=-\int_{\R^N}F(|x|,u_1,\ldots,u_m)\,dx,
\end{equation}
and $E^3$ represents a nonlocal contribution of the form
\begin{equation}
	\label{0.03}
	E^3(U)=-\int_{\R^N}\int_{\R^N}G\left(u_1(x), \ldots, u_m(x)\right)V(|x-y|)G\left(u_1(y),\ldots,u_m(y)\right)\,dy. 
\end{equation}
The proof of \eqref{0.0} and \eqref{0.1i} is based on polarization techniques. 
Indeed, under suitable supermodularity assumptions on $F$, i.e., \textbf{(F 3)} in Section~\ref{sec:app},
we have that 
\begin{equation}
	\label{0.4}
	E^2(u_1,\ldots,u_m)\leq E^2\left(u_1^H, \ldots, u_m^H\right) \quad\forall H\in\mathcal{H},
\end{equation}
see \cite{BuHa06a}. (In fact, quite recently, it was proved by the first author~\cite{H3} that
supermodularity is also necessary for \eqref{0.4}.)
If $G$ satisfies a related assumption, i.e., \textbf{(G 4)} in Section~\ref{sec:app}, and $V$ is a non-increasing kernel, we observe 
in Proposition~\ref{prop:Grearr} that
\begin{equation}
	\label{0.5}
	E^3(u_1,\ldots,u_m)\leq E^3\left(u_1^H, \ldots, u_m^H\right) \quad\forall H\in\mathcal{H}.
\end{equation}
On the other hand: For any $u\in L^p(\R^N)$, it was established in~\cite{BroSo00} 
that there exists a sequence $u_n$, obtained by iterated polarizations of $u$ with respect 
to some appropriate closed half-spaces $H_n\in \mathcal{H}$, such that $u_n\to u^*$ in $L^p(\RR^N)$,
see Theorem~\ref{thm:BroSo} below.
Moreover, 
\eqref{0.4} (or \eqref{0.5}) together with this approximation of the Schwarz rearrangement enable us to conclude that 
\begin{equation}
	\label{0.7}
	E^2(u_1,\ldots,u_m)\leq E^2\left(u_1^*, \ldots, u_m^*\right) 
\end{equation}
and 
\begin{equation}
	\label{0.8}
	E^3(u_1,\ldots,u_m)\leq E^3\left(u_1^*, \ldots, u_m^*\right). 
\end{equation}
Thus, the cases of equality in~\eqref{0.7} and \eqref{0.8}
reduce to the less difficult identities
\begin{equation}
	\label{0.11}
	E^2(u_1,\ldots,u_m)= E^2\left(u_1^H, \ldots, u_m^H\right) \quad\forall H\in\mathcal{H}
\end{equation}
and 
\begin{equation}
	\label{0.12}
	E^3(u_1,\ldots,u_m)= E^3\left(u_1^H, \ldots, u_m^H\right) \quad\forall H\in\mathcal{H}.
\end{equation}
In case of $E^2$, this problem was completely solved in~\cite{BuHa06a}, where it is proved that 
under strict supermodularity assumptions (\textbf{(F 3)} with strict inequalities), 
\eqref{0.11} is equivalent to $U=U^*$.

Therefore, we already have suitable conditions ensuring 
that whenever \eqref{PC} has a minimizer, each component is radial and radially decreasing 
(up to a translation). While these results have 
many relevant applications in economics (\cite{Carlier} and the references therein) and 
physics (\cite{BuHa06a}, \cite{H3} and the references therein),
in some important contexts, the strict supermodularity of $F$ is not a plausible assumption. 
More precisely, the profile of stable electromagnetic waves traveling along a 
planar wave-guide are given by the ground states of the energy functional
$$
L(u)=\frac{1}{2}\int u'^2-\int F(|x|,u)\,dx
$$
under the constraint $|u|_2=c$, $|x|$ is the position relative to the optical axis. 
$F$ is determined by the index of refraction of the media and $c>0$ is a parameter 
related to the wave speed~\cite{St1}. The wave-guide is composed from different layers (core and claddings),
and the index of refraction is a non-increasing function with respect to the distance $|x|$, 
but it is constant in each layer in the most relevant situations. 
Therefore $F$ cannot be a stricly supermodular function. Nevertheless, experiments done by engineers 
show that the ground state is Schwarz symmetric (up to a translation). 

For such local non-linearities and a $G$ which is supermodular but not strictly so, 
we are compelled to study the cases of equality in the following, more complicated rearrangement inequality: 
\begin{equation}
	\label{0.13}
	\int_{\R^N}j(u,\abs{Du})\,dx\geq \int_{\R^N}j(u^*,\nabs{Du^*})\,dx,
\end{equation}
called the generalized Polya-Szeg\"o inequality. We are then looking for reasonable assumptions on $j$ under which 
equality in~\eqref{0.13} implies (roughly) $u=u^*$ (up to translations). 
This is carried out in Section 2, where we show that equality in \eqref{0.13}
reduces to equality in the standard Polya-Szeg\"o inequality
\begin{equation}
	\label{0.18}
	\int_{\R^N} \abs{Du}^p\,dx\geq \int_{\R^N} \nabs{Du^*}^p \,dx,
\end{equation}
which in turn was completely solved by Brothers and Ziemer \cite{BroZie88a}.

Let us also point out that in~\cite{HS1a, HS1b}, cases of equality in the generalized Polya-Szeg\"o inequality 
were established under the ``fact'' that the integrand~$j$ is such that equality in \eqref{0.13} 
implies that $u_n\rightarrow u^*$ in $W^{1,p}$ for the sequence $(u_n)$ of iterated polarizations of $u$ approximating $u^*$ in $L^p$ (hypothesis (3.4) of Theorem 3.5 in \cite{HS1b}), which of course is not a very tangible assumption. 

For a single constraint, this deficiency was resolved in~\cite{HS2}, 
but only under some restrictive assumptions on~$j$ and its derivatives (Theorem 1.1 in~\cite{HS2}). 
The idea developed there is based on the fact that due to the rearrangement inequalities, 
we know that if~$u$ is a solution of \eqref{PC}, then the sequence of iterated polarizations $(u_n)$ is also a solution, i.e, 
$E(u_n)=I$ and $u_n\in S \quad\forall n\in\N$. 
Thus, each term $u_n$ satisfies an Euler-Lagrange equation obtained by tools of non-smooth analysis. 
As a second step, one shows that the corresponding Lagrange multipliers $\lambda_n$ form a bounded sequence in~$\R$. 
Almost everywhere convergence of the sequence of gradients $Du_n$ to $Du^*$ then follows 
by applying the result of Dal Maso and Murat \cite{DaMu98a} to an appropriate sequence of Leray-Lions type operators 
associated to $j(u_n,|Du_n|)$. 
Ultimately, this leads to $\norm{Du}_{L^p(\R^N)}=\nnorm{Du^*}_{L^p(\R^N)}$.
Apart from the fact that this approach imposes a lot of technical hypotheses on~$j$, 
it does not easily extend to the case of multiple constraints because of the second step.

Among other things, we obtain symmetry properties of the ground state solutions of a system of Euler-Lagrange equations 
associated to our functional, see \eqref{EL} below.
The establishment of symmetry properties of such solutions of \eqref{EL} 
constitutes in itself a branch in analysis. Numerous papers have been addressed to this issue, 
especially when $m=1, g=0, j(t,s)=s^2$ and $F(r,s_1)=h(s_1)$. For $m>1$, 
the determination of the symmetry of the solutions of (2.1) has applications in the theory of Bose-Einstein condensates, 
optical pulse propagation in birefringent fiber, interactions of m-wave packets (see~\cite{H1} and the references therein). 
The high degree of difficulty encountered by researches is due to the fact that most of the tools used 
for scalar equations do not extend to $m>1$. Until quite recently, most of the articles dealt with the autonomous case 
$F(r,s_1,s_2)=\frac{1}{2k}s_1^{2k}+\frac{1}{2k}s_2^{2k}+\frac{\beta}{k}s_1^{k}s_2^{k}$, 
$g\equiv 0, j(t,s)=s^2$.
In~\cite{H4}, the first author extended these results to general~$m>1$ and~$F$ but with $j$ and~$G$ as before. 
Note that the case $G\not\equiv 0$ is important since it describes 
many interesting situations in which we have a Coulomb-type interactions between particles. 
Here, we study the symmetry of ground state solutions of \eqref{EL} under very general assumptions.


\section{Weak-strong convergence and symmetry of minimizers via iterated polarizations}

\subsection{Polarization and Schwarz rearrangement: Basic facts}

Let $H\subset \RR^N$ be a (closed) half-space, i.e., $H=\{x\in \RR^N\mid x\cdot e\geq t\}$ for some $e\in S^{N-1}$ and $t\in \RR$,
and let $u^H:\RR^N\to \RR$ denote the polarization of a function $u:\RR^N\to \RR$ with respect $H$, i.e.,
$$
	u^H:=\left\{\begin{array}{ll}
		\max\{u,v\}&\text{on $H$,}\\
		\min\{u,v\}&\text{on $\RR^N\setminus H$,}
	\end{array}	\right.
	\text{where}~v(x):=u(x_H)~\text{for $x\in \RR^N$}
$$
and $x_H$ denotes the reflection of $x$ with respect to the hyperplane $\partial H$. 
On functions with values in $\RR^m$, the polarization operates component-wise.
If $u\in W^{1,p}(\RR^N)$, then $u^H\in W^{1,p}(\RR^N)$ and 
$$
	D u^H=\left\{\begin{array}{ll}
		D u&\text{a.e.~on}~(\{u>v\}\cap H)\cup (\{u<v\}\setminus H),\\
		D u=D v&\text{a.e.~on}~\{u=v\},\\
		D v&\text{a.e.~on}~(\{u<v\}\cap H)\cup (\{u>v\}\setminus H).
	\end{array}\right.
$$
Using this, it is not difficult to check the following invariance of homogeneous integral functionals under polarization:
\begin{lem}[\cite{HS1b}, e.g.]\label{lem:polintinv}
Suppose that 
$I:=\int_{\RR^N} f(u,\nabs{D u})\,dx$ is well defined and finite for some $f:\RR\times \RR_+\to \RR$ and some $u\in W^{1,p}(\RR^N)$.
Then for every half-space $H$, $I^H:=\int_{\RR^N} f(u^H,\nabs{D u^H})\,dx$ is well defined and finite, and $I^H=I$.
\end{lem}
The Schwarz rearrangement (or Schwarz symmetrization) of a measurable function $u:\RR^N\to \RR_+:=[0,\infty)$ is defined 
as the radially symmetric, radially non-increasing function $u^*:\RR^N\to \RR_+$ such that
$\abs{\{u>h\}}=\abs{\{u^*>h\}}$ for every $h>0$. Note that $u^*$ is unique up to changes on a set of measure zero.
Polarizations and Schwarz symmetrization are linked as follows:
\begin{thm}[\cite{BroSo00}]\label{thm:BroSo}
Let $1\leq p<\infty$, $u_0\in L^{p}(\RR^N)$, $u_0\geq 0$, let $H_n\subset \RR^N$ be a sequence of half-spaces {\bf containing the origin},
and let $(u_n)$ denote the associated sequence of iterated polarizations, i.e., $u_{n}:=u_{n-1}^{H_n}$.
Then $(u_n)$ is relatively compact in $L^p(\RR^N)$. 
Moreover, for a suitable choice of the sequence $H_n$, $u_n\to u^*$ in $L^p(\RR^N)$.
\end{thm}
\begin{rem}\label{rem:polandsymm}
The proof in \cite{BroSo00} also provides an inductive rule for the choice of an appropriate sequence of half-spaces which at first glance depends on $p$. 
Nevertheless, if $u_0\in L^p\cap L^q$ for some $q\in [1,\infty)$, $q\neq p$, and the half-spaces $H_n$ are such that $u_n\to u_0^*$ in $L^p$, 
then also $u_n\to u_0^*$ in $L^q$, due to the relative compactness of $u_n$ in this space which is independent of the sequence $H_n$.
Also note that any dense sequence $H_n$ works as shown in \cite{H1} and \cite{VS}, which means that $(H_n)$ can be chosen independently of $u_0$. 
\end{rem}

\subsection{Partial compactness for the gradients of iterated polarizations}
As already observed in \cite{BroSo00}, the compactness of a sequence of iterated polarizations in $W^{1,p}$ is directly related to the Schwarz symmetry of the initial function. In particular, compactness of the gradients of a sequence of iterated polarizations cannot be expected in general. Nevertheless, we are able to obtain a partial result in terms of the following notion of equiintegrability that is well suited for the use on domains with infinite measure:
\begin{defn}[Equiintegrability]
For $1\leq r<\infty$, an open set $\Omega\subset\RR^N$ and a sequence $(U_n)\subset L^r(\Omega)^m$, we say that $(U_n)$ is $r$-equiintegrable if the following two properties hold:
\begin{equation}\label{requiint}
\begin{aligned}
	&\sup \mysetl{\int_E \nabs{U_n}^r\,dx}{n\in\NN, E\subset \Omega,~\abs{E}\leq \delta}\underset{\delta\to 0^+}{\To} 0, \\
	&\sup \mysetl{\int_{\Omega\setminus B_R} \nabs{U_n}^r\,dx}{n\in\NN}\underset{R\to\infty}{\To} 0.
\end{aligned}
\end{equation}
Here, $B_R\subset \RR^N$ denotes a ball of radius $R$ centered at zero.
\end{defn}
Note that for any domain, convergent sequences in $L^r$ are $r$-equiintegrable. Conversely, $r$-equiintegrable sequences are bounded in $L^r$, and the only remaining obstacle to relative compactness in $L^r$ are possible oscillations.

Usually, compactness of a sequence of functions in $L^p$ or related properties do not provide any information about compactness of the associated sequence of gradients (if they exist at all).
However, in the special case of iterated polarizations of a fixed function, the situation is different:
\begin{lem}\label{lem:polgradequiint}
Let $p,q\in [1,\infty)$, let $H_n$ be a sequence of half-spaces in $\RR^N$, let 
$u_0\in W^{1,1}_\loc$ such that $u\in L^q(\RR^N)$ and $D u\in L^p(\RR^N)^N$,
and let $u_n=u_{n-1}^{H_n}$ for $n\in\NN$ be the associated sequence of iterated polarizations of $u_0$. 
Then $q$-equiintegrability of $(u_n)$ implies
$p$-equiintegrability of $(D u_n)$.
\end{lem}
\begin{rem}
This also closes a gap in the proof of Theorem 8.2 in \cite{BroSo00} for $p=1$, where $1$-equiintegrability of the gradients has to be shown but the proof of
the second part of \eqref{requiint} is missing.
\end{rem}
\begin{proof}
Let $E\subset \RR^N$ be measurable and $n\in \NN$.
Since $u_n$ and $\abs{D u_n}$ are obtained by repeatedly simultaneously rearranging $u_0$ and $\abs{D u_0}$, respectively, 
there exist rearrangements $E_n$ of $E$ with $\abs{E_n}=\abs{E}$ such that
$$
	\int_E \abs{u_n}^{q}=\int_{E_n} \abs{u_0}^{q}~~\text{and}~~\int_E \abs{D u_n}^p=\int_{E_n} \abs{D u_0}^p.
$$
Moreover, if $E=u_n^{-1}(A)$ or $E=(\abs{D u_n})^{-1}(A)$ for some set $A\subset \RR$, then $E_n$ can be chosen as
$E_n=u^{-1}(A)$ or $E=(\abs{D u_n})^{-1}(A)$, respectively.
In particular, 
$$
	\int_{\{\abs{D u_n}< \delta\}} \abs{D u_n}^p\,dx=
	\int_{\{\abs{D u}< \delta\}} \abs{D u}^p\,dx\underset{\delta\to 0^+}{\To}0
$$
and
$$
	\int_{\{\abs{D u_n}>h\}} \abs{D u_n}^p\,dx=
	\int_{\{\abs{D u}>h\}} \abs{D u}^p\,dx\underset{h\to\infty}{\To}0
$$
uniformly in $n$.
It thus suffices so show that for every $0<\delta<h<\infty$,
\begin{align}\label{pgei-1}
	\sup_n\int_{\{\delta\leq  \abs{D u_n}\leq h\}\setminus B_R} \abs{D u_n}^p\,dx
	\underset{R\to\infty}{\To}0.
\end{align}
If $E_n=E_n(R)$ denotes the set associated to $E:=B_R$ as above, we have that
\begin{align}\label{pgei-2}
	\int_{\{\delta\leq \abs{D u_n}\leq h\}\setminus B_R} \abs{D u_n}^p\,dx
	=
	\int_{\{\delta\leq \abs{D u}\leq h\}\setminus E_n(R)} \abs{D u}^p\,dx.
\end{align}
Suppose now that there exists an $r_0>0$ such that
$$
	0<\mu:=\frac{1}{2}\limsup_{R\to\infty}\sup_n \abs{\big(\{\delta\leq \abs{D u}\leq h\}\setminus E_n(R)\big)\cap B_{r_0}}.
$$
(Otherwise, \eqref{pgei-1} immediately follows from \eqref{pgei-2}, because the constant sequence $D u$ is $p$-equiintegrable.)
In particular, the measure of
$$
	U_{r,\delta,h}:=B_r\cap \{\delta\leq \abs{D u}\leq h\}
$$
is greater than $\mu$ for every $r\geq r_0$.
We claim that for every $r\geq r_0$, there exists a constant $C_r=C_r(u,\mu,p,q,\delta,h)>0$ such that
\begin{align}\label{pgei-3}
 \int_F\abs{D u}^p\leq C_r \int_F \abs{u}^{q}
	~~\text{for every}~F\subset U_{r,\delta,h}~\text{with}~\abs{F}\geq \mu.
\end{align}
For a proof, observe that 
\begin{align}\label{pgei-4}
	0<\frac{1}{C_r}:=\frac{1}{\abs{U_{r,\delta,h}} h^{p}} \inf_F \int_F \abs{u}^{q}
	\leq \inf_F \frac{\int_F \abs{u}^{q}}{\int_F \abs{D u}^p}<\infty
\end{align}
Here, the first infimum, taken over $F\subset U_{r,\delta,h}$ with $\abs{F}\geq \mu$,
is indeed attained and thus positive:
Cleary, $U_{r,\delta,h}\subset \{\abs{u}>0\}$, at least up to a set of measure zero (recall that $\abs{\{v=0\}\setminus \{D v=0\}}=0$ for any $v\in W^{1,1}_\loc$).
Moreover, $\abs{U_{r,\delta,h}\cap \{0<\abs{u}\}}> \mu$ and  $\abs{U_{r,\delta,h}\cap \{0<\abs{u}<s\}}\to 0$ as $s\to 0^+$, whence
there exists a level $s_0\in (0,\infty)$ such that
$$
	\abs{U_{r,\delta,h}\cap \{0<\abs{u}<s_0\}}\leq \mu~~\text{and}~~
	\abs{U_{r,\delta,h}\cap \{0<\abs{u}\leq s_0\}}\geq \mu.
$$	
With this choice of $s_0$, we have that $\inf_F \int_F \abs{u}^{q}=\int_{F_0} \abs{u}^{q}$ for any measurable $F_0$ such that $\abs{F_0}=\mu$ and
$U_{r,\delta,h}\cap \{0<\abs{u}<s_0\}\subset F_0 \subset U_{r,\delta,h}\cap \{0<\abs{u}\leq s_0\}$.

As a consequence of \eqref{pgei-2}, \eqref{pgei-3} and the definition of $E_n(R)$,
\begin{align*}
	&\int_{\{\delta\leq \abs{D u_n}\leq h\}\setminus B_R} \abs{D u_n}^p\\
	&\leq
	\int_{\{\delta\leq \abs{D u}\leq h\}\setminus E_n(R)\setminus B_r} \abs{D u}^p
	+C_r \int_{\{\delta\leq \abs{D u}\leq h\}\setminus E_n(R)\cap B_r} \abs{u}^{q}\\
	&\leq 
	\int_{\RR^N\setminus  B_r} \abs{D u}^p\,dx
	+C_r \int_{\{\delta\leq \abs{D u}\leq h\}\setminus E_n(R)} \abs{u}^{q}\\
	&=	\int_{\RR^N\setminus  B_r} \abs{D u}^p\,dx
	+C_r \int_{\{\delta\leq \abs{D u_n}\leq h\}\setminus B_R} \abs{u_n}^{q}\\
	&\leq	\int_{\RR^N\setminus  B_r} \abs{D u}^p\,dx
	+C_r \int_{\RR^N\setminus B_R} \abs{u_n}^{q}
\end{align*}
Since $(u_n)$ is $q$-equiintegrable, $\sup_n C_r \int_{\RR^N\setminus B_R} \abs{D u_n}^{q}\to 0$ as $R\to\infty$ for every fixed $r$,
and $\int_{\RR^N\setminus  B_r} \abs{D u}^p\to 0$ as $r\to\infty$.
Combined, this implies \eqref{pgei-1}.
\end{proof}

To get full compactness of $D u_n$ in $L^p$, additional properties of the initial function $u_0$ are needed. For minimizers of an integral functional,
the following theorem on weak-strong convergence turns out to be useful. It relies on equiintegrability to replace otherwise necessary 
coercivity assumptions on the integrand of the functional.

\subsection{An adapted theorem on weak-strong convergence}

Let $1<p<\infty$ and let $\Omega\subset \RR^N$ be open (possibly unbounded) and smooth enough such that $W^{1,p}(\Omega)$ is continuously embedded in $L^{p^*}(\Omega)$.
Here $p^*=\frac{pN}{N-p}$ is the critical Sobolev exponent if $p<N$ (otherwise, $p^*\in (p,\infty)$ can be chosen arbitrarily but fixed). For $u\in W^{1,p}(\Omega)$, we consider the functional
defined by
$$
 	J(u):=\int_\Omega j(x,u,D u)\,dx\in \RR\cup\{+\infty\},
$$
where
\begin{align}
\label{j0}
	&j:\Omega\times \RR\times \RR^N\to \RR~~\text{is a Carath\'eodory function\footnotemark.}
\end{align}%
\footnotetext{i.e., $j=f(x,\mu,\xi)$ is measurable in $x\in\Omega$ for every $(\mu,\xi)$ and continuous in $(\mu,\xi)\in\RR\times \RR^N$ for a.e.~$x$}%
In addition, we need that
\begin{align}
	\label{j1}
	&j(x,\mu,\xi)\geq -C\big(\abs{\xi}^p+\abs{\mu}^{p^*}+\abs{\mu}^{p}\big)-\abs{h(x)},\\
	\label{j2}
	&j(x,\mu,\cdot)~~\text{is strictly convex},
\end{align}
with a constant $C>0$ and $h\in L^1(\Omega)$, for a.e.~$x\in \Omega$, every $\mu\in \RR$ and every $\xi\in \RR^N$. 
Then we have the following theorem, essentially a variant of the results of {\sc Visintin} \cite{Vi84a}, {\sc Evans \& Gariepy} \cite{EvGa87a}, {\sc Zhang} \cite{Zha89a} and {\sc Sychev} \cite{Sy98a,Sy99a} for the scalar case on unbounded domains (for related results on unbounded domains also see \cite{Kroe09cp}).
\begin{thm}\label{thm:weakstrong}
Assume that \eqref{j0}--\eqref{j2} hold, let $(u_n)\subset W^{1,p}(\Omega)$ be a bounded sequence weakly converging to a function $u\in W^{1,p}(\Omega)$ and 
suppose that
\begin{align}\label{thsws-0}
	\text{$(u_n)$ is $p$- and $p^*$-equiintegrable and $(D u_n)$ is $p$-equiintegrable}. 
\end{align}
Then $\liminf J(u_n)\geq J(u)$.
If, in addition, $\limsup J(u_n)\leq J(u)<\infty$, then $u_n\to u$ strongly in $W^{1,1}_\loc(\Omega)$, and as a consequence of \eqref{thsws-0}, $D u_n\to D u$ in $L^p(\Omega;\RR^N)$ and $u_n\to u$ in $(L^{p^*}\cap L^p)(\Omega)$. 
\end{thm}

%
The results cited above are not applicable in the situation of Theorem~\ref{thm:weakstrong}, in particular because of the unbounded domain and the extremely weak lower bound. Nevertheless, it is not difficult to obtain a proof following the approach of \cite{Sy98a,Sy99a} or \cite{Kroe09cp} based on the theory of Young measures. In fact, since we only consider the scalar case (in particular, quasiconvexity reduces to convexity) and exploit the equiintegrability of the sequence considered, it can be simplified significantly. 
Our starting point is the fundamental theorem for Young measures:
\begin{thm}[\cite{Ba89a,Mue99a}, e.g.]\label{thm:Ym}
Let $w_n:\Omega\to \RR^m$ be a sequence of measurable functions. Then there exists a subsequence $(w_{k(n)})$ and
a family $\nu=(\nu_x)_{x\in \Omega}$ of non-negative Radon measures on $\RR^m$, 
weak*-measurable\footnote{i.e., $x\mapsto \int _{\RR^m}f(\mu) d\nu_x(\mu)$ is measurable for every $f\in C_0(\RR^m)$} \,in $x$, such that the following holds:
\begin{myenum}
\item[(i)]
	$\nu_x(\RR^m) \leq 1$ for a.e.~$x\in \Omega$.
\item[(ii)]
	If $\lim_{h\to\infty}\sup_{n\in \NN} 
	\measN{\{\nabs{w_{k(n)}}\geq h\}\cap \Omega\cap B_R(0)}=0$ for every $R>0$,\\
	then $\nu_x(\RR^m) = 1$ for a.e.~$x\in \Omega$.
\item[(iii)]	
	For every Carath\'eodory function $f:\Omega\times \RR^m\to \RR$ such that\\
	$f(\cdot,w_{k(n)})$ is $1$-equiintegrable\footnotemark, 
	we have that\\
	$\displaystyle{\int_\Omega f(x,w_{k(n)}(x))\,dx\underset{n\to \infty}{\To}
	\int_\Omega \int_{\RR^m} f(x,\mu)\,d\nu_x(\mu)dx}$.
\end{myenum}
Moreover,
$w_n$ converges locally in measure to some function $w$ if and only if $\nu_x$ is supported on the singleton $\{w(x)\}$ for a.e.~$x$.%
\footnotetext{Note that $1$-equiintegrablility is equivalent to weak relative compactness in $L^1$ by the de la Vall\'e-Poussin criterion.}%
\end{thm}
As a consequence of (iii), $\nu$ is a.e.~uniquely determined by $(w_{k(n)})$. It is called the Young measure generated by $w_{k(n)}$.
\begin{rem}
If $w_n\to w$ weakly in $L^p_\loc$, then
$\int_{\RR^m} \xi d_{\nu_x}(\xi)=w(x)$ as a consequence of (iii). 
\end{rem}
\begin{proof}[Proof of Theorem~\ref{thm:weakstrong}]
We focus on the proof of second part of the assertion; the first part on weak lower semicontiunity (which is well known anyway, cf.~\cite{Io77a}) is obtained as a byproduct. It suffices to show that every subsequence of $u_n$ has another subsequence that converges to $u$ in $W^{1,1}_\loc(\Omega)$.
Hence, we may assume w.l.o.g.~that $u_n\to u$ in $L^p_\loc(\Omega)$ and that $(D u_n)$ generates a Young measure $\nu_x$, which for a.e.~$x$ is a probability measure on $\RR^N$ by Theorem~\ref{thm:Ym} (ii). 
Since $u_n\to u$ in $L^p_\loc$, $(u_n)$ generates the Young measure $\delta_{u(x)}$, the Dirac mass concentrated at $u(x)$. As a consequence,
the Young measure generated by $w_n:=(u_n,D u_n)$ is given by $\delta_{u(x)}\otimes \nu_x$.
In the following, for $h>0$ consider the truncated integrands
$$
	j^{[h]}(x,\mu\,\xi):=\chi_{B_h(0)}(x)\min\{h,j(x,\mu,\xi)\},~~x\in\Omega,~\mu\in \RR,~\xi\in\RR^N,
$$
with $\chi_{B_h(0)}$ denoting the characteristic function of the ball given in the index.
By \eqref{j1}, \eqref{thsws-0} implies $1$-equiintegrability of $j^{[h]}(x,u_n,D u_n)$, and Theorem~\ref{thm:Ym} (iii) yields that
$$
	J(u_n)\geq \int_\Omega j^{[h]}(x,u_n,D u_n)\,dx
	\underset{n\to\infty}{\To}
	\int_\Omega \int_{\RR^N} j^{[h]}(x,u(x),\xi)\,d\nu_x(\xi) dx
$$
In the limit $h\to \infty$, this entails that
\begin{align}\label{tYm-1}
	\liminf_n J(u_n)\geq 
	\int_\Omega \int_{\RR^N} j(x,u(x),\xi)\,d\nu_x(\xi) dx.
\end{align}
On the other hand, by the convexity of $j$ in its last variable, Jensen's inequality yields that
\begin{align}\label{tYm-2}
	\int_{\RR^N} j(x,u(x),\xi)\,d\nu_x(\xi)
	\geq 
	j\Big(x,u(x),\textstyle{\int_{\RR^N}} \xi d\nu_x(\xi)\Big)
	=j\big(x,u(x),D u(x)\big) 
\end{align}
for a.e.~$x$. 
Combined, \eqref{tYm-1} and \eqref{tYm-2} imply that $\liminf J(u_n)\geq J(u)$, and by assumption, we actually have equality. Hence, Jensen's inequality \eqref{tYm-2} also holds with equality for a.e.~$x$. Due to the strict convexity assumed in \eqref{j2}, the latter is the case if 
and only if $\nu_x$ is a Dirac mass, i.e., $\nu_x=\delta_{D u(x)}$.
In particular, $D u_n\to D u$ locally in measure, and since $u_n\to u$ in $L^p_\loc$ and $(D u_n)$ is $1$-equiintegrable, this entails that
$u_n\to u$ in $W^{1,1}_\loc$ as claimed.
\end{proof}
\begin{rem}
The method used in the proof of Theorem~\ref{thm:weakstrong} works equally well in the fully coupled vector case, i.e., for functionals of the form $\int j(x,U,DU)$ with vector-valued $U$. In this setting, the assumption of strict convexity in the gradient variable can even be replaced by a suitable notion of strong quasi-convexity, at least if $j$ also satisfies a growth condition. However, we do not know of any reasonably general assumptions on $j$ that guarantee a rearrangement inequality with respect to (component-wise) polarization or Schwarz symmetrization 
for such integrals, which prevents the kind of application we have in mind here.
\end{rem}

\subsection{Symmetry of minimizers}

We consider functionals of the form
$$
	E(U):=\sum_{i=1}^m J_i(u_i)-K(U),~~E:S\to \RR,~~J_i(u_i):=\int_{\RR^N} j_i(u_i,\abs{D u_i})\,dx,
$$
where $U=(u_1,\ldots,u_m)\in S$, constrained to a set $S\subset W^{1,p}_+(\RR^N)^m$ such that
$$
	\text{$S$ is closed in $(L^p\cap L^{p^*})(\RR^N)^m$ and invariant under polarizations,}
$$
and thus, by Theorem~\ref{thm:BroSo}, also invariant under Schwarz rearrangement.	
In addition, we assume that
for each $i=1,\ldots,m$,
\begin{align}
	&j_i:\RR\times \RR_+\to\RR~~\text{is continuous},\label{ji0}\\
	&j_i(s,t)\geq -C(\abs{s}^{p^*}+\abs{s}^p+t^p),\label{ji1}\\
	&j_i(s,\cdot)~~\text{is non-decreasing and strictly convex,}\label{ji2}
\end{align}
for every $s\in\RR$ and $t\in \RR_+$, with constants $C>0$ and $1<p<\infty$, and
\begin{align}
	\label{K0}
	&K:S\to \RR~~\text{is continuous with respect to the topology of $(L^{p^*}\cap L^p)^m$,}\\
	\label{K1}
	&K(U)\leq K(U^*)~~\text{for every $U\in S$.}
\end{align}
Here, $p^*:=\frac{pN}{N-p}$ if $p<N$, while $p^*\in [p,\infty)$ can be chosen arbitrarily (but fixed) if $p\geq N$.
We obtain the following result related to the symmetry of functions $U$ satisfying the generalized Polya-Szeg\"o inequality $E(U^*)\leq E(U)$ with equality:
\begin{thm}\label{thm:minsym}
Assume that $1<p<\infty$ and that \eqref{ji0}--\eqref{K1} hold. Moreover, suppose that there is a function $U=(u_1,\ldots,u_m)\in S$ 
such that $E(U)\leq E(U^*)<\infty$. 
Then $J_i(u_i)=J_i(u_i^*)$ and $\norm{D u_i}_{L^p}=\norm{D u_i^*}_{L^p}$ for $i=1,\ldots,m$.
\end{thm}
\begin{cor}\label{cor:minsym}
If, in addition to the assumptions of Theorem~\ref{thm:minsym}, the set 
$$
	C_i^*:=\{D u_i^*=0\}\cap \{u_i^*\in (0,\esssup u_i)\}
$$ 
has measure zero 
(or, equivalently, $t\mapsto \abs{\{u_i>t\}}$ is absolutely continuous on $(0,\esssup u_i)$), then up to a translation,
$u_i=u_i^*$ and thus is radially symmetric and radially non-increasing. 
\end{cor}
\begin{rem}
In~\cite{HS1a, HS1b}, a related result was established, but \emph{assuming} that the functional has a weak-strong convergence property. 
\end{rem}
\begin{rem}
Clearly, every minimizer $U\in S$ satisfies $E(U)\leq E(U^*)$.
\end{rem}
\begin{proof}[Proof of Theorem~\ref{thm:minsym}]
Let $(H_n)$ be a sequence of half-spaces in $\RR^N$ containing the origin, and let $U^n=(u^n_1,\ldots,u^n_m)$ be the associated sequence of iterated polarizations of $U$, i.e.,
$U^0:=U$ and $U^{n}:=(U^{n-1})^{H_n}$.
Since $U\in (L^p\cap L^{p^*})^m$, we have that 
$U^n\to U^*$ in $(L^p\cap L^{p^*})^m$, at least for an appropriate choice of $(H_n)$, see~Theorem~\ref{thm:BroSo} and Remark~\ref{rem:polandsymm}.
In particular, 
\begin{align}\label{tms-1}
	K(U^n)\to K(U^*)
\end{align}	
by \eqref{K0}. 
Moreover, by Lemma~\ref{lem:polintinv}, $\int j_i(u_i^n,\abs{D u_i^n})=\int j_i(u_i,\abs{D u_i})$, $\int \abs{u_i^n}^p=\int \abs{u_i}^p$ and 
$\int \abs{D u_i^n}^p=\int \abs{D u_i}^p$,
whence $U^n$ is bounded in $(W^{1,p})^m$, $U^n\rightharpoonup U^*$ weakly in $(W^{1,p})^m$ and $U^*\in S$.
In addition, Lemma~\ref{lem:polgradequiint} yields that $D u_i^n$ is $p$-equiintegrable.
By the weak lower semicontinuity of $J_i$ along $u_i^n$ obtained in Theorem~\ref{thm:weakstrong}, we have 
\begin{align}\label{tms-2}
	J_i(u_i)=\lim_n J_i(u_i^n)\geq J_i(u_i^*),  
\end{align}
and by \eqref{K1}, this implies that $E(U)\geq E(U^*)$. Since the converse inequality holds by assumption, we even have that $E(U)=E(U^*)$ and $E(U^n)\to E(U^*)$.
In view of \eqref{tms-1} and \eqref{tms-2}, the latter is possible only if
\begin{align*}
	J_i(u_i)=\lim_n J_i(u_i^n)=J_i(u_i^*).  
\end{align*}
By Theorem~\ref{thm:weakstrong}, we conclude that
$u_i^n\to u_i^*$ strongly in $W^{1,p}$, and in particular,
$\norm{D u_i}_{L^p}=\lim_n\norm{D u_i^n}_{L^p}=\norm{D u_i^*}_{L^p}$.
\end{proof}
\begin{proof}[Proof of Corollary~\ref{cor:minsym}]
By the main result of \cite{BroZie88a}, $\norm{D u_i}_{L^p}=\norm{D u_i^*}_{L^p}$ implies the assertion.
\end{proof}

\section{Applications to minimization problems}\label{sec:app}
For $m\in \NN$, $1<p<\infty$ and $U=(u_{1},..., u_{m}) \in W^{1,p}(\mathbb{R}^{N})^m$, we study minimization problems of the following form:
\begin{eqnarray}\label{Mc}
M_{c}:=\inf\left\{ E(U) \mid U\in S_{c}\right\},
\end{eqnarray}
\begin{eqnarray*}
S_{c}:=\left\{ U=(u_{1},..., u_{m}) \in W^{1,p}(\mathbb{R}^{N})^{m}\,\left|\, \int |u_{i}|^{p} =c_{i},~1\leq i\leq m\right.\right\},
\end{eqnarray*}
where $c=(c_{1},...,c_{m})\in \RR^m$ is a prescribed vector with positive components and $E$ is the functional defined by
\begin{eqnarray*}
E(U)&=& E(u_{1},..., u_{m}):=\sum_{i=1}^{m}\int J_{i}(u_{i},|D u_{i}|)- \int F(|x|,u_{1}(x),..., u_{m}(x)) \\
&-& \int \int G(u_{1}(x),..., u_{m}(x)) V(|x-y|) G(u_{1}(y),..., u_{m}(y)) \: dx \: dy
\end{eqnarray*}
Here and below, integrals whose domain is not specified as a subscript are always taken over $\RR^N$.
\begin{rem}
Under suitable regularity assumptions on $J_{i},\: F$ and $ G$, minimizers of (\ref{Mc}) yield a nontrivial solutions of the quasilinear system of equations given by
\begin{equation} \label{EL}
\begin{aligned}
 \div \big( D_{\xi} J_{i}(u_{i},|D u_{i}|) \big)= & D_{s} J_{i}(u_{i},|D u_{i}|)
 + D{s_{i}} F(|x|,u_{1},..., u_{n})\\
 &+2(V \ast G) G_{s_{i}}(u_{1},..., u_{m}) + \lambda_{i} u_{i} |u_{i}|^{p-2},
\end{aligned}
\end{equation}
$1\leq i\leq m$, for some Lagrange multipliers $\lambda_{i} \in \mathbb{R}$.
\end{rem}
Now let us state our assumptions on  $J_{i},\: F$ and $ G$. 
Via the Sobolev embedding, the critical
exponent $p^*:=\frac{pN}{N-p}$ comes into play for $p<N$; if $p\geq N$, $p^*\in(p,\infty)$ can be chosen arbitrarily below. \\
\textbf{Assumptions on} $\mbf{J_{i}}$:\\
For $1\leq i\leq m$, $J_{i}: \mathbb{R}\times \mathbb{R}_{+}\longrightarrow \mathbb{R}_+$ is a continuous function such that: 
\begin{enumerate}
\item[\textbf{(J 0)}]
	$J_{i}(|s|, b) \leq J_{i}(s, b)$ for all $s\in \mathbb{R},\: b\in \mathbb{R}_{+}$;
\item[\textbf{(J 1)}]
	$\exists \;a_{1}>0$ s.~t.~$J_{i}(s, b)\geq  a_{1}b^{p}$ for any $1\leq i\leq m$, $s\in \mathbb{R}_{+},\: b\in \mathbb{R}_{+}$;
\item[\textbf{(J 2)}]
	$\forall \;1\leq i\leq m$, $J_{i}(s,.)$ is convex and non-decreasing for all $s\in \mathbb{R}_{+}$;
\end{enumerate}
\textbf{Assumptions on} $\mbf{F}$:\\
$F: \mathbb{R}_{+}\times \mathbb{R}^{m}\longrightarrow \mathbb{R}$ is a Carath\'eodory function, i.e.,
\begin{enumerate}
\item[(i)]
  $F(\cdot, s): \mathbb{R}_{+}\to \mathbb{R}$ is measurable in $\mathbb{R}_{+}\setminus \Gamma$ for all $s\in \RR^m$, where $\Gamma$ is a subset of $\mathbb{R}_{+}$
	having one dimensional measure zero, and
\item[(ii)]
  for every $r \in \mathbb{R}_{+}\setminus \Gamma$, the function
	$\mathbb{R}^m \to  \mathbb{R}$, $s=(s_1,\ldots,s_m)\mapsto F (r,s)$, is continuous. 
\end{enumerate}
In addition, we assume the following:
\begin{enumerate}
\item[\textbf{(F 0)}]
	$F(r, s_{1},...,s_{m})\leq F(r, |s_{1}|,...,|s_{m}|)$ for a.e.~$r\geq 0$ and every $s_{1},...,s_{m}\in \mathbb{R}$;
\item[\textbf{(F 1)}]
	for a.e.~$r\geq 0$ and every $s_{1},...,s_{m}\geq 0$,
 	$$0\leq F(r, s_{1},...,s_{m}) \leq K\Big( |s|^{p}+ \sum_{i=1}^{m} s_{i}^{l_{i}+p}\Big),~~s=(s_{1},...,s_{m}),$$ 
 	with positive constants $K$ and $0<l_{i}< \tfrac{p^{2}}{N}$;
\item[\textbf{(F 2)}]
	for every $\varepsilon >0$ there exist $R_{0}>0$ and $s_{0}>0$ such that\\
	$
		F(r, s_{1},...,s_{m})\leq \varepsilon |s|^{p}
	$
	for a.e.~$r\geq R_{0}$, $0\leq 	s_{1},...,s_{m} < s_{0}$;
\item[\textbf{(F 3)}]	$(t,y)\mapsto F(\frac{1}{t},y)$ is supermodular on $\RR_+\times \RR_+^m$, i.e.,
	for a.e.~$r\geq 0$, every $y\in \RR_+^m$, every $h,k\geq 0$, every $i\neq j$, $i,j=1,\ldots,m$, 
	and a.e.~$R\geq r$,
	$$F(r,y+h e_i +k e_j)+F(r,y)\geq F(r,y+h e_i)+F(r,y+k e_j)$$
	and
	$$F(r,y+h e_i)+F(R,y)\geq F(R,y+h e_i)+F(r,y),$$
	where $e_i$ denotes the $i$-th unit vector in $\RR^m$. 
\end{enumerate}
\textbf{Assumptions on $\mbf{G}$ and the  Coulomb type potential $\mbf{V}$}:\\
$ G :\mathbb{R}^{m}\longrightarrow \mathbb{R}_+$  is continuous and $ V: \mathbb{R}_+\longrightarrow \mathbb{R}_{+}$ is measurable such that: 
\begin{enumerate}
\item[\textbf{(G 0)}]
	$G(s_{1},...,s_{m})\leq G( |s_{1}|,...,|s_{m}|)$ for all $s_{1},...,s_{m}\in \mathbb{R}$;
\item[\textbf{(G 1)}]
	there exists a positive constant $K'$ such that
	$$0\leq G(s_{1},...,s_{m})\leq K' \sum_{i}^{m} s_{i}^{\mu_{i}}~~\text{for all $s_{1},...,s_{m}\geq 0$},$$
	where $p \tfrac{2q-1}{2q}< \mu_{i} < p^{*} \tfrac{2q-1}{2q}$ 
	and $1<q<\infty $ depends on $V$ as follows:
\item[\textbf{(G 2)}]
	$V: \mathbb{R}_{+}\longrightarrow \mathbb{R}_{+}$, $V(\abs{\cdot})\in L^q_w(\RR^N)$ with $N 
		 \Big(\frac{\mu_i}{p}-\frac{2q-1}{2q}\Big)<\frac{p}{2}$, $1\leq i\leq m$;
\item[\textbf{(G 3)}]
	$0\leq V(|x|)\leq V(|y|)$ \; $\forall |x|\geq |y|$;
\item[\textbf{(G 4)}] $G:\RR_+^m\to \RR_+$ is non-decreasing in each variable and supermodular, i.e., 
$$
	G(y+h e_i)\geq G(y),
$$
$$
	G(y+h e_i +k e_j)+G(y)\geq G(y+h e_i)+G(y+k e_j),
$$
for every $y\in \RR_+^m$, every $h,k\geq 0$ and every $i\neq j$, $i,j=1,\ldots,m$,
where $e_i$ denotes the $i$-th unit vector in $\RR^m$.
\end{enumerate}
Here, $L^q_w$ denotes the weak $L^q$ space consisting of all measurable functions
$u$ for which $\sup_{h>0}h^{q}\abs{\{\abs{u}\geq h\}}$ is finite.

Finally, we need an additional assumption to make up for the fact that our constraint set $S_c$ is not compact in the topology of $W^{1,p}(\RR^N)^m$ and to ensure that $M_c$ is attained. For this purpose,
we define
$$
	\tilde{M}_c:=\inf\left\{ E(u_1,\ldots,u_m) \left| 
	\begin{array}{ll}
	u_i\in W^{1,p}_+(\RR^N), \int \abs{u_i}^p\leq c_i, 
	i=1,\ldots,m
	\end{array} \right. \right\}
$$
consider some $c=(c_1,\ldots,c_m)\in (0,\infty)^m$ such that
\begin{enumerate}
\item[\textbf{(E 0)}]
	$\tilde{M}_c<\tilde{M}_d$ for every $d=(d_1,\ldots,d_m)\in (0,\infty)^m$ such that $d_i\leq c_i$ for every $i\in \{1,\ldots,m\}$
	and $d_k<c_k$ for one $k\in\{1,\ldots,m\}$.
\end{enumerate}
\begin{rem} Note that by definition, $c\mapsto \tilde{M}_c$ is non-increasing in every component of $c$.
The strict monotonicity at one $c$ required in \textbf{(E 0)} is essentially independent of the rest of our assumptions above. 
It is just some sufficient condition one can use to show that $M_c$ is attained and  it does not play a role in our symmetry results. 
\end{rem}

Under the assumptions on $J_i$, $F$, $G$ and $V$ listed above, we obtain
\begin{thm}\label{thm:app}
Let $c\in (0,\infty)^m$ and suppose that $M_c<\infty$. Then we have the following:
\begin{enumerate}
\item[(i)]
If \textbf{(E 0)} holds, 
$M_{c}=\inf\left\{ E(U); U\in S_{c}\right\}$ is attained in $S_c$, and at least one minimizer is radially symmetric and has non-negative components. 
\item[(ii)] 
Suppose that the convexity of $J_i(s,\cdot)$ required in \textbf{(J 2)} is strict for some $i\in \{1,\ldots,n\}$. Then
for every minimizer $U=(u_1,\ldots,u_m)\in (W^{1,p}_+)^m\cap S_c$ of \eqref{Mc}, 
we have that $\norm{Du_i}_{L^p}=\norm{Du_i^*}_{L^p}$. If, in addition,
\begin{align*} 
	\qquad C^*_i:=\{D u_i^*=0\}\cap \{u_i^*\in (0,\esssup u_i)\}\subset \RR^N~~\text{has measure zero},
\end{align*}	
then $u_i=u_i^{*}$ up to a translation. 
\end{enumerate}

\end{thm}
\begin{rem}
As illustrated by some examples in \cite{BroZie88a}, the condition $\nabs{C^*_i}=0$ cannot be dropped in general. Of course, special properties of the functional might still imply this for minimizers $U$, in particular if $U$ happens to be analytic.
\end{rem}
\begin{ex}
For instance, our assumptions are satisfied for 
$$
	\begin{aligned}
		E(U):=&\int_{\RR^3} \sum_{i=1}^m \Big(1+\frac{1}{1+\abs{u_i(x)}}\Big)\abs{D u_i(x)}^2 \,dx\\
		&-\int_{\RR^3}\int_{\RR^3} \Big(\sum_{i=1}^m \abs{u_i(x)}^2\Big) \frac{1}{\abs{x-y}}\Big(\sum_{j=1}^m \abs{u_j(y)}^2\Big)\,dxdy
	\end{aligned}
$$
with $N=3$, $p=2$ and $q=3$.
Moreover, using the monotonicity and homogeneity properties of the integrands in the variable $U$, 
it is not difficult to see that \textbf{(E 0)} holds if $M_c<0$.
The latter is satisfied for every $c\in (0,\infty)^m$; in fact, a simple calculation yields that $E(U_\delta)<0$ for $\delta>0$ small enough, where $U_\delta(x):=\delta^{\frac{3}{2}}U(\delta x)$ for some arbitrary, fixed $U\in S_c$ (whence $U_\delta\in S_c$ for every $\delta>0$). 
\end{ex}

\subsection{Auxiliary results: Continuity and compactness of the lower order terms}

We first recall the following well known compact embedding for Sobolev spaces of radial functions:
\begin{lem}[see \cite{SuWaWi07a}, e.g.]\label{lem:radcompact}
Let $\cU \subset W^{1,p}(\RR^N)$ be a bounded set of radially symmetric functions (with respect to some fixed point in $\RR^N$).
Then $\cU$ is relatively compact in $L^s(\RR^N)$ provided $p<s<p^*=\frac{pN}{N-p}$. 
(If $p\geq N$ set $p^*=+\infty$ in this context.)
\end{lem}

The local perturbation has the following properties:
\begin{prop}\label{prop:J2compact}
Suppose that \textbf{(F 1)} and \textbf{(F 2)} hold.
Then 
$$
	U\mapsto F(\abs{\cdot},U),~~(L^p\cap L^{p^*})(\RR^N)^m\to L^1(\RR^N),~~\text{is continuous},
$$
and 
$$
	U\mapsto F(\abs{\cdot},U),~~W^{1,p}_\rad(\RR^N)^m\to L^1(\RR^N),~~\text{is compact}
$$ 
(i.e., maps bounded subsets of $(W^{1,p}_\rad)^m$, the subspace of radially symmetric functions in $(W^{1,p})^m$, to relatively compact subsets of $L^1$).
\end{prop}
\begin{proof}
The continuity of the Nemytskii operator $U\mapsto F(\abs{\cdot},U)$ associated to $F$ is standard for a Carath\'eodory function satisfying the growth condition \textbf{(F 1)}; here, note that $\frac{p^2}{N}+p\leq p^*$. For the second part of the assertion let $(U_n)=((u_{n,1},\ldots,u_{n,m}))\subset W^{1,p}_\rad(\RR^N)^m$ be a bounded sequence.
Let $\eps>0$, let $R_0>0$ and $0<s_0\leq 1$ be the associated constants in \textbf{(F 2)} and consider the set
$$
	A_\eps=A_\eps(n):=\{\abs{x}\geq R_0\}\cap \{\abs{U_n(x)}\leq s_0\}
$$
Observe that by \textbf{(F 1)} and \textbf{(F 2)},
\begin{equation*}
	\int_{A_\eps} \abs{F(\abs{x},U_n)}\,dx\leq \eps \int_{\RR^N} \abs{U_n}^p\,dx\leq C\eps,
\end{equation*}
with a constant $C$ independent of $n$. Moreover, by \textbf{(F 1)},
$U\mapsto F(\abs{\cdot},U),~~(L^p\cap L^{p+l})\to L^1$ is continuous, where $l:=\max_i l_i$,
and it thus suffices to show that for every fixed $\eps>0$,
\begin{equation}\label{pJ2c-2}
	\text{$(U_n)$ is relatively compact in $(L^p\cap L^{p+l})(\RR^N\setminus A_\eps;\RR^m)$}.
\end{equation}
By Lemma~\ref{lem:radcompact}, we immediately get that 
$(U_n)$ is relatively compact in $(L^s\cap L^{p+l})(\RR^N)^m$
for every $p<s\leq p+l$ since $p\leq p+l<p^*$ by \textbf{(F 1)}.
By H\"older's inequality and the fact that $\RR^N\setminus A_\eps$ has finite measure, this implies \eqref{pJ2c-2}.
\end{proof}

For the convolution term in the energy, we rely on Young's inequality for convolutions in weak form: 
If $v\in L_w^q(\RR^N)$, 
$f\in L^s(\RR^N)$, $g\in L^{t}(\RR^N)$, $s,t\in (1,\infty]$, $q\in (1,\infty)$ and $\frac{1}{q}+\frac{1}{s}+\frac{1}{t}=2$
then $f\cdot(v*g)$ belongs to $L^1(\RR^N)$ and
\begin{equation}\label{Yiconv}
	\norm{f(v*g)}_{L^1}\leq C \norm{W}_{L^q_w}\norm{f}_{L^s}\norm{g}_{L^t}
\end{equation}
holds with a constant $C=C(q,s,t,N)$, cf.~\cite{LiLo01B}. Here,
$v*g$ denotes the convolution of $v$ and $g$.


\begin{prop}\label{prop:convtermwcont}
Let $1<q<\infty$, let $v\in L^q_w(\RR^N)$, let
$1<\alpha\leq \beta <\infty$ and suppose that
$2-\frac{2}{\alpha}\leq \frac{1}{q}$ and
$2-\frac{2}{\beta}\geq \frac{1}{q}$.
Then
\begin{align}\label{pctwc-1}
	(g,h)\mapsto g(v*h),~~(L^\alpha\cap L^\beta)(\RR^N)\times (L^\alpha\cap L^\beta)(\RR^N)\to L^1(\RR^N),
\end{align}
is continuous.
\end{prop}
\begin{proof}
We split
$$
	v=v_1+v_2~~\text{with}~v_1:=\chi_{\{\abs{v}<1\}}v,~v_2:=\chi_{\{\abs{v}\geq 1\}}v.
$$
By assumption, $s:=\frac{\alpha}{2\alpha-2}\geq q$, whence $v_1\in L^s_w$.
Young's inequality for convolutions combined with H\"older's inequality yields that
\begin{alignat*}{2} 
	&\norm{g(v_1*h)}_{L^1}\leq \norm{v_1}_{L^s_w}\norm{g}_{L^\alpha}\norm{h}_{L^\alpha}.
\end{alignat*}
Similarly, we have that $t:=\frac{\beta}{2\beta-2}\leq q$, whence $v_2\in L^t_w$ and
\begin{alignat*}{2} 
	&\norm{g(v_2*h)}_{L^1}\leq \norm{v_2}_{L^t_w}\norm{g}_{L^\beta}\norm{h}_{L^\beta}.
\end{alignat*}
Together, this implies that
\begin{alignat*}{2} 
	&\norm{g(v*h)}_{L^1}\leq \big(\norm{v_1}_{L^s_w}+\norm{v_2}_{L^t_w}\big) \norm{g}_{L^\alpha\cap L^\beta}\norm{h}_{L^\alpha\cap L^\beta}.
\end{alignat*}
The asserted continuity follows since $(g,h)\mapsto g(v*h)$ is bilinear.
\end{proof}
For the convolution term in our functional, we get
\begin{prop}\label{prop:J3compact}
Suppose that \textbf{(G 1)} and \textbf{(G 2)} hold.
Then 
$$
	U\mapsto G(U)[V(\abs{\cdot})*G(U)],~~(L^p\cap L^{p^*})(\RR^N)^m\to L^1(\RR^N),~~\text{is continuous},
$$
and 
$$
	U\mapsto G(U)[V(\abs{\cdot})*G(U)],~~W^{1,p}_\rad(\RR^N)^m\to L^1(\RR^N),~~\text{is compact.}
$$ 
\end{prop}
\begin{rem}
For compactness, the restriction to radially symmetric functions cannot be dropped because
$U\mapsto G(U)[v(\abs{\cdot})*G(U)]$ is invariant under translations.
\end{rem}
\begin{proof}
In the following, we abbreviate
$\mu_-:=\min_j \mu_j$ and $\mu_+:=\max_j \mu_j$.
Using the continuity and growth of $G$, we have that $U\mapsto G(U)$, $(L^p\cap L^{p^*})^m\to L^{\alpha}\cap L^{\beta}$, is continuous
provided that
\begin{equation*}
	\max\{1,\tfrac{p}{\mu_-}\}\leq \alpha\leq \beta\leq  \tfrac{p^*}{\mu_+}~~\text{($\max\{1,\tfrac{p}{\mu_-}\}\leq \alpha\leq \beta<\infty$ if $p\geq N$)}.
\end{equation*}
Moreover, by Lemma~\ref{lem:radcompact}, $W_\rad^{1,p}$ is compactly embedded in $L^s\cap L^t$ with $p<s\leq t<p^*$, and
similarly as before, we find that $U\mapsto G(U)$, $W^{1,p}_\rad\to L^{\alpha}\cap L^{\beta}$ is compact whenever 
\begin{equation}\label{pJ3c-2}
	\max\{1,\tfrac{p}{\mu_-}\}< \alpha\leq \beta< \tfrac{p^*}{\mu_+}~~\text{($\max\{1,\tfrac{p}{\mu_-}\}< \alpha\leq \beta< \infty$ if $p\geq N$)}.
\end{equation}
As a consequence of the inequalities relating $\mu_j$, $p$ and $p^*$ in \textbf{(G 1)},
there exist $\alpha,\beta$ such that in addition to \eqref{pJ3c-2},
$2-\frac{2}{\alpha}\leq \frac{1}{q}$ and $2-\frac{2}{\beta}\geq  \frac{1}{q}$.
By Proposition~\ref{prop:convtermwcont}, the continuity and compactness of $U\mapsto G(U)$, respectively, now
carry over to the asserted continuity and compactness
$U\mapsto G(U)[V(\abs{\cdot})*G(U)]$.
\end{proof}

\subsection{Auxiliary results: Rearrangement inequalities}

The leading part of the energy satisfies the following extended Polya-Szeg\"o inequality:
\begin{prop}\label{prop:Jrearr}
Let $i\in \{1,\ldots m\}$ and suppose that $J_i:\RR_+\times \RR_+\to \RR_+$ is a continuous function satisfying \textbf{(J 2)}.
Then for every $u_i\in W^{1,p}_+(\RR^N)$,
\begin{equation}\label{Jrearr}
	\int_{\RR^N} J_i(u_i^*,\nabs{Du_i^*})\,dx\leq \int_{\RR^N} J_i(u_i,\nabs{Du_i})\,dx.
\end{equation}
\end{prop}
\begin{proof}
This is essentially well known. It is not difficult to obtain a proof based on the approximation of $u_i^*$ by iterated polarizations of $u_i$ (Theorem~\ref{thm:BroSo} and Remark~\ref{rem:polandsymm}), 
Lemma~\ref{lem:polintinv} and the weak lower semicontinuity of $u\mapsto \int_{\RR^N} J_i(u,\nabs{Du})$ in $W^{1,p}$ shown in Theorem~\ref{thm:weakstrong}. We omit the details.
\end{proof}
The local and nonlocal perturbations behave as follows under polarization and Schwarz rearrangement: 
\begin{prop}\label{prop:Frearr}
Suppose that $F:\RR_+\times \RR_+^m\to \RR$ is a Caratheodry function satisfying \textbf{(F 1)} and \textbf{(F 3)}, and let 
$U=(u_1,\ldots,u_m)\in (L^p\cap L^{p^*})(\RR^N)^m$ such that $u_i\geq 0$ for $i=1,\ldots,m$.
Then for every every closed half-space $H\subset \RR^N$ containing the origin, we have that
$$
	\int_{\RR^N} F(\abs{x},U^H)\,dx~\geq~ \int_{\RR^N} F(\abs{x},U)\,dx,
$$
where $U^H=(u_1^H,\ldots,u_m^H)$ is the polarization of $U$ with respect to $H$. Moreover,
$$
	\int_{\RR^N} F(\abs{x},U^*)\,dx~\geq~ \int_{\RR^N} F(\abs{x},U)\,dx,
$$
where $U^*=(u_1^*,\ldots,u_m^*)$ is the Schwarz rearrangement of $U$.
\end{prop}
\begin{proof}
See \cite{BuHa06a}.
\end{proof}
\begin{prop}\label{prop:Grearr}
Suppose that $G:\RR_+^m\to \RR^+$ is a continuous function, suppose that $G$ and $V$ satisfy \textbf{(G 1)}--\textbf{(G 4)} and let
$U=(u_1,\ldots,u_m)\in (L^p\cap L^{p^*})(\RR^N)^m$ such that $u_i\geq 0$ for $i=1,\ldots,m$.
With
$$
	Q(U):=\int_{\RR^N} \int_{\RR^N} G(u_{1}(x),\ldots, u_{m}(x)) V(|x-y|) G(u_{1}(y),\ldots, u_{m}(y)) \: dx \: dy,
$$
we then have that
$$
	Q(U^H)\geq Q(U)
$$
for every closed half-space $H\subset \RR^N$, where $U^H=(u_1^H,\ldots,u_m^H)$ denotes the polarization of $U$ with respect to $H$. In addition,
$$
	Q(U^*)\geq Q(U),
$$
where $U^*=(u_1^*,\ldots,u_m^*)$ is the Schwarz rearrangement of $U$.
\end{prop}
\begin{proof}
First observe that \textbf{(G 4)} implies the supermodularity of
$$
	(v_1,\ldots,v_m,w_1,\ldots,w_m)\mapsto G(v_1,\ldots,v_m)G(w_1,\ldots,w_m),~~\RR_+^{2m}\to \RR_+.
$$
(In fact, the converse is also true, since by \textbf{(G 1)}, $G$ cannot be decreasing in each component unless $G\equiv 0$.)
As a consequence, we have that
\begin{equation*}
\begin{aligned}
	&G(v_1,\ldots,v_m)G(w_1,\ldots,w_m)+G(v'_1,\ldots,v'_m)G(w'_1,\ldots,w'_m)\\
	&\leq \begin{aligned}[t]
		&G(\max\{v_1,v'_1\},\ldots,\max\{v_m,v'_m\})G(\max\{w_1,w'_1\},\ldots,\max\{w_m,w'_m\})\\
		&+G(\min\{v_1,v'_1\},\ldots,\min\{v_m,v'_m\})G(\min\{w_1,w'_1\},\ldots,\min\{w_m,w'_m\})
	\end{aligned}
\end{aligned}
\end{equation*}
for every $v,v',w,w'\in \RR^m_+$ (see Lemma 6.1 in \cite{BuHa06a}). In view of the definition of $u^H$, this means that for every $x,y\in H$,
\begin{equation}\label{pconvrearr-2}
	\begin{aligned}
	&G(U(x))G(U(y))+G(U(x_H))G(U(y_H))\\
	&\qquad \leq G(U^H(x))G(U^H(y))+G(U^H(x_H))G(U^H(y_H))
	\end{aligned}
\end{equation}
and similarly, the supermodularity of $G$ yields that
\begin{equation}\label{pconvrearr-3}
	G(u(x))+G(U(x_H))\leq G(U^H(x))+G(U^H(x_H)).
\end{equation}
Here, $x_H$ is the reflection of $x$ with respect to $\partial H$.
Splitting each of the two integrals in the definition of $Q$ into integrals over $H$ and its complement, a change of variables yields
\begin{align*}
	Q(U)&=\begin{aligned}[t]
	&\int_H\int_H G(U(x))V(\abs{x-y})G(U(y))\,dxdy\\
	&+\int_H\int_H G(U(x_H))V(\abs{x_H-y_H})G(U(y_H))\,dxdy\\
	&+\int_H\int_H G(U(x))V(\abs{x-y_H})G(U(y_H))\,dxdy\\
	&+\int_H\int_H G(U(x_H))V(\abs{x_H-y})G(U(y))\,dxdy,
	\end{aligned}
\end{align*}
Using \eqref{pconvrearr-2} and \eqref{pconvrearr-3} together with the fact that 
$$
	A(x,y):=V(\abs{x-y})=V(\abs{x_H-y_H}) ~\geq~ a(x,y):=V(\abs{x-y_H})=V(\abs{x_H-y}),
$$
we infer that
\begin{align*}
	&Q(U)\\
	&=\begin{aligned}[t]
	&\int_H\int_H \big[G(U(x))G(U(y))+G(U(x_H))G(U(y_H))\big](A-a)(x,y)\,dxdy\\
	&+\int_H\int_H \big[G(U(x))+G(U(x_H))\big]\big[G(U(y))+G(U(y_H))\big] a(x,y)\,dxdy
	\end{aligned}\\
	&\leq \begin{aligned}[t]
	&\int_H\int_H \big[G(U^H(x))G(U^H(y))+G(U^H(x_H))G(U^H(y_H))\big](A-a)(x,y)\,dxdy\\
	&+\int_H\int_H \big[G(U^H(x))+G(U^H(x_H))\big]\big[G(U^H(y))+G(U^H(y_H))\big] a(x,y)\,dxdy
	\end{aligned}\\
	&= \begin{aligned}[t]
	Q(U^H).
	\end{aligned}
\end{align*}
The corresponding inequality for the Schwarz rearrangement,
$
	Q(U^*)\geq Q(U),
$
is now a consequence of the approximation of $U^*$ by a sequence of iterated polarizations (Theorem~\ref{thm:BroSo} and Remark~\ref{rem:polandsymm}),
also exploiting the continuity of $Q$ in $(L^p\cap L^{p*})^m$ due to Proposition~\ref{prop:J3compact}.
\end{proof}

\subsection{Proof of Theorem~\ref{thm:app}}

The proof is divided into five steps.

\verb"Step I": \eqref{Mc} \verb"is well posed "($M_{c}>-\infty$)\verb", and minimizing sequences"\\ 
\verb"are bounded in "$(W^{1,p})^m$.\\
By \textbf{(F 0)} and \textbf{(F 1)} we can write
\begin{eqnarray*}
\int F(|x|,u_{1}(x),..., u_{m}(x)) \: dx &\leq & \int F(|x|,|u_{1}(x)|,..., |u_{m}(x)|)\: dx \\
&\leq& K c + K \sum_{i=1}^{m} \int |u_{i}(x)|^{l_{i}+p}.
\end{eqnarray*}
Now for $1\leq i\leq m$, the Gagliardo-Nirenberg inequality tells us that
$$\norm{u_{i}}_{l_{i}+p}\leq K^{''}\norm{u_{i}}_{p}^{1-\sigma_{i}} \norm{D u_{i}}_{p}^{\sigma_{i}}\;   \text{with}\; \sigma_{i}=\frac{N}{p} \frac{l_{i}}{l_{i}+p}.$$
For $\varepsilon>0$ set $p_{i}=\frac{p^{2}}{N l_{i}} $ and $q_{i}$ such that $\frac{1}{p_{i}} + \frac{1}{q_{i}} =1$, by Young's inequality we obtain that
$$
	\norm{u_{i}}_{l_{i}+p}^{l_{i}+p}
	~\leq~ 
	\left({\frac{K^{''}}{\varepsilon}}^{l_{i}+p}\norm{u_{i}}_{p}^{(1-\sigma_{i})(l_{i}+p)}\right) ^{q_{i}}\frac{1}{q_{i}}
+ \frac{N l_{i}}{p^{2}} \left( \varepsilon ^{\frac{N l_{i}}{p^{2}}}\norm{D u_{i}}_{p}^{p}\right).
$$
On the other hand, it follows from \textbf{(G 0)} that
\begin{align*}
	&\int \int G(u_{1}(x),..., u_{m}(x)) V(|x-y|) G(u_{1}(y),..., u_{m}(y)) \: dx \: dy\\ 
	&\leq \int \int G(|u_{1}(x)|,..., |u_{m}(x)|) V(|x-y|) G(|u_{1}(y)|,..., |u_{m}(y)|) \: dx \: dy.
\end{align*}
Then using \textbf{(G 1)}, \textbf{(G 2)} and Young's inequality for convolutions \eqref{Yiconv}, we get that
\begin{equation*}
\begin{aligned}
\int \int G(u_{1}(x),..., u_{m}(x)) V(|x-y|) G(u_{1}(y),..., u_{m}(y)) \: dx \: dy&\\
 \leq ~ {K^{'}}^{2} \sum_{i,j=1}^{m} \norm{V}_{L^q_w} \norm{\nabs{u_{i}}^{\mu_{i}}}_{q^{'}}\norm{\nabs{u_{j}}^{\mu_{j}}}_{q^{'}}&
\end{aligned}
\end{equation*}
where $q^{'}=\frac{2q}{2q-1}$.
Therefore, the Gagliardo-Nirenberg inequality yields that
$$
	\norm{\nabs{u_{i}}^{\mu_{i}}}_{q^{'}}\leq K^{''} \norm{u_{i}}_{p}^{(1-\gamma_{i})\mu_{i}} \norm{D u_{i}}_{p}^{\gamma_{i}\mu_{i}},
$$
and thus
$$
	\norm{\nabs{u_{i}}^{\mu_{i}}}_{q^{'}}\norm{\nabs{u_{j}}^{\mu_{j}}}_{q^{'}}
	\leq 
	{K^{''}}^2 \norm{V}_{L^{q}_{w}}\norm{u_{i}}_{p}^{(1-\gamma_{i})\mu_{i}} \norm{u_{j}}_{p}^{(1-\gamma_{j})\mu_{j}}
 \norm{D u_{i}}_{p}^{\gamma_{i}\mu_{i}}  \norm{D u_{j}}_{p}^{\gamma_{j}\mu_{j}}.
$$
where $\gamma_{i} =\frac{N}{p}\frac{\frac{2q\mu_{i}}{2q-1}-p} {\frac{2q\mu_{i}}{2q-1}}$.
Setting 
$$
	\alpha_{ij}:=\frac{p}{\gamma_i\mu_i+\gamma_j\mu_j}
	\text{and}~~
	\Gamma_{ij}:=(1-\gamma_{i}) \mu_{i}+(1-\gamma_{j}) \mu_{j}
$$
we certainly have that $\alpha_{ij}(\gamma_{i}\mu_{i}+\gamma_{j}\mu_{j})=p$. 
By Young's inequality, we infer that
\begin{align*}
 \int \int G(u_{1}(x),..., u_{m}(x)) V(|x-y|) G(u_{1}(y),..., u_{m}(y)) \: dx \: dy &\\
 ~\leq ~ {K^{'}}^2\sum_{i,j=1}^m \Big(
 \frac{{K^{''}}^2}{\alpha_{ij}^{'}}\left(\frac{1}{\eps} (\norm{u}_p)^{\Gamma_{ij}}\right)^{\alpha_{ij}^{'}} + \frac{1}{\alpha_{ij}} \eps^{\alpha_{ij}} \norm{D u}_{p}^{p} \Big)&
\end{align*}
for every $\eps>0$, where $\alpha_{ij}^{'}:=\frac{\alpha_{ij}}{\alpha_{ij}-1}$. Here, note that $\alpha_{ij}>1$ for $i,j=1,\ldots,m$ due to the inequality in \textbf{(G 2)}.
Since $\norm{u}_p$ is bounded by a constant only depending on $c$, we may summarize
\begin{align*}
 \int \int G(u_{1}(x),..., u_{m}(x)) V(|x-y|) G(u_{1}(y),..., u_{m}(y)) \: dx \: dy&\\
 \leq ~ K^{'''}_\eps+\frac{m^2}{\alpha} \eps^{\alpha} \norm{D u}_{p}^{p}&.
\end{align*}
Finally using \textbf{(J 0)} and \textbf{(J 1)} and gathering all the terms, we can write:
\begin{align*}
E(U) ~\geq~   \left( a_{1} - \frac{m^{2}}{\alpha} \eps^{\alpha}- K \sum_{i=1}^{m}\frac{N l_{i}}{p^{2}} \eps ^{\frac{p^{2}}{N l_{i}}}\right)\norm{D u_{i}}_{p}^{p}&\\
 - \left(\sum_{i=1}^{m}  \frac{{
 K^{''}}^{l_{i}+p}}{\eps} c^{(1-\sigma_{i})(\frac{l_{i}}{p}+1)}\right) \frac{1}{q_{i}} - K^{'''}_\eps&,
\end{align*}
Choosing $\eps$ small enough in such a way that
$$a_{1} - \frac{m^{2}}{\alpha} \eps^{\alpha}- K \sum_{i=1}^{m}\frac{N l_{i}}{p^{2}} \eps ^{\frac{p^{2}}{N l_{i}}}~>~ 0$$
enables us to conclude that $M_{c}>-\infty$. In addition, we get that any minimizing sequence is bounded in $ (W^{1,p})^{m}$.
As we need this below, note that by the same argument, $\tilde{M}_c>-\infty$ and minimizing sequences for $\tilde{M}_c$ are bounded in $(W^{1,p})^m$.

\verb"Step II": \verb"Existence of a Schwarz symmetric minimizing sequence"\\
We claim that there exists a minimizing sequence $U_{n}=(u_{n,1},...,u_{n,m})$ of our variational problem (\ref{Mc}) which is Schwarz symmetric, i.e.
$0\leq u_{n,i}=u_{n,i}^{*}$ for $1\leq i\leq m$.\\
First note that if $u \in W^{1,p}(\mathbb{R}^{N})$ then $|u| \in W^{1,p}(\mathbb{R}^{N})$, for instance see \cite{LiLo01B}.\\
Now by virtue of \textbf{(J 0)}, \textbf{(F 0)} and \textbf{(G 0)}, we obviously have that
\begin{eqnarray}\label{moduleq}
& &E(|u_{1}|,..., |u_{m}|) \leq E(u_{1},..., u_{m})= E(U)\\
& & \forall U=(u_{1},..., u_{m}) \in W^{1,p}(\mathbb{R}^{N})^{m}.\nonumber
\end{eqnarray}
By Proposition~\ref{prop:Jrearr}, we know that
\begin{eqnarray}\label{adjleq}
\forall \;1\leq i\leq m \;\; \int J_{i}(u,|D u|)\geq \int J_{i}(u^{*},|D u^{*}|)\;\; \forall \: u  \in W^{1,p}(\mathbb{R}^{N})
\end{eqnarray}
On the other hand by the Cavalieri's principle, we have that
\begin{eqnarray}\label{adjequal}
\norm{u}_{p}=\norm{u^{*}}_{p}\;\;\; \forall \: u  \in L^{+}_{p}(\mathbb{R}^{N})
\end{eqnarray}
which implies in our context that if $U=(u_{1},..., u_{m}) \in S_{c}$ then \\ $(|u_{1}|^{*},..., |u_{m}|^{*}) \in S_{c}$.\\
As a consequence of \textbf{(F 0)}, \textbf{(G 0)}, Proposition~\ref{prop:Frearr}, Proposition~\ref{prop:Grearr} and the approximation of $(|u_{1}|^{*},..., |u_{m}|^{*})$ by sequences of iterated polarizations, 
exploiting the continuity in $L^p\cap L^{p^*}$ of the functionals involved,
we obtain the following rearrangement inequalities: 
\begin{eqnarray*}
\int F(|x|,u_{1}(x),..., u_{m}(x)) \: dx \leq  \int F(|x|, |u_{1}|^{*}(x),..., |u_{m}|^{*}(x))\: dx
\end{eqnarray*}
and
\begin{align*}
&\int \int G(u_{1}(x),..., u_{m}(x)) V(|x-y|) G(u_{1}(y),..., u_{m}(y)) \: dx \: dy\\ 
&\leq \int \int G(|u_{1}|^{*}(x),..., |u_{m}|^{*}(x)) V(|x-y|) G(|u_{1}|^{*}(y),..., |u_{m}|^{*}(y)) \: dx \: dy,
\end{align*}
for every $u_{1},...,u_{m} \in (L^{p}\cap L^{p^*})(\mathbb{R}^{N})$.\\
In conclusion, 
we have
\begin{align}\label{E leq}
E(|u_{1}|^{*},..., |u_{m}|^{*}) \leq E(|u_{1}|,..., |u_{m}|) \leq E(u_{1},..., u_{m})&\\
\text{for any} \;u_{1},...,u_{m} \in W^{1,p}(\mathbb{R}^{N})&.\nonumber
\end{align}

\verb"Step III": \verb"Lower semi-continuity along radial sequences"\\
Let $U_{n}=(u_{n,1},...,u_{n,m})=(u^{*}_{n,1},...,u^{*}_{n,m})=U_{n}^{*}$ be a Schwarz symmetric sequence with non-negative components. In this step we show that under
\textbf{(J 1)}, \textbf{(J 2)}, \textbf{(F 1)}, \textbf{(F 2)}, \textbf{(G 1)}, \textbf{(G 2)} and  \textbf{(G 3)}, the following holds:\\
If \; $U_{n}=U_{n}^{*}\rightharpoonup U \;\; in \;\; W^{1,p}(\mathbb{R}^{N})^{m}$  with some
$U=(u_{1},...,u_{m}) \in W^{1,p}_{+}(\mathbb{R}^{N})$, then
\begin{eqnarray}\label{E leq liminf}
E(U) \leq  \liminf E(U_{n}).
\end{eqnarray}
For the proof of \eqref{E leq liminf}, first observe that by virtue of \textbf{(J 2)}, it follows from the weak lower semi-continuity of $J_i$ 
shown in Theorem~\ref{thm:weakstrong} that
\begin{eqnarray}\label{ui lim inf}
\forall \;1\leq i\leq m \;\; \int J_{i}(u_{i},|D u_{i}|)\leq  \liminf \int J_{i}(u_{n,i},|D u_{n,i}|)
\end{eqnarray}
In addition, we have that
\begin{eqnarray}\label{int unj1}
\lim_{n\rightarrow \infty} \int F(|x|,u_{n,1}(x),..., u_{n,m}(x))=  \int F(|x|,u_{1}(x),..., u_{m}(x))
\end{eqnarray}
and
\begin{eqnarray}\label{int unj2}
\lim_{n\rightarrow \infty} \int \int G(u_{n,1}(x),..., u_{n,m}(x)) V(|x-y|) G(u_{n,1}(y),..., u_{n,m}(y)) \: dx \: dy \nonumber \\
= \int \int G(u_{1}(x),..., u_{m}(x)) V(|x-y|) G(u_{1}(y),..., u_{m}(y)) \: dx \: dy,
\end{eqnarray}
by the second part of Proposition~\ref{prop:J2compact} and Proposition~\ref{prop:J3compact}, respectively. Combined, \eqref{ui lim inf}--\eqref{int unj2} yield \eqref{E leq liminf}. 

\verb"Step IV:" \verb"If" \textbf{(E 0)} \verb"holds," $M_{c}$ \verb"is attained".\\
Let $U=(u_{1},\ldots,u_{m})$ be the weak limit of a Schwarz-symmetric minimizing sequence $U_n=(u_{n,1},\ldots,u_{n,m})\in W^{1,p}_+(\RR^N)^m$ for $\tilde{M}_c$
so that $c_{n,i}:=\int \abs{u_{n,i}}^p\leq c_i$ for $i=1,\ldots,m$ and $E(U_n)\to \tilde{M}_c$.
Here, recall that by the first step, $U_n$ is bounded in $W^{1,p}$ and thus weakly converges up to a subsequence.
Due to the previous step, we know that $E(U) \leq \tilde{M}_{c}$ by \eqref{E leq liminf}, and by the properties of the norm $\norm{\cdot}_{p}$,
\begin{eqnarray} \label{bounded u i}
	d_i:=\norm{u_{i}}_{p}^{p} \leq c_{i} \hspace*{2cm} \forall \; 1\leq i\leq m.
\end{eqnarray}
Using \textbf{(E 0)}, we infer that $d=(d_1,\ldots,d_m)=c$, because otherwise,
$$
	\tilde{M}_{c}< \tilde{M}_{d}\leq E(U)\leq \tilde{M}_{c},
$$
which in impossible. In particular, $U=U^*\in S_c$ and $M_c\leq E(U)=\tilde{M}_{c}\leq M_c$, whence $M_c$ is attained.

\verb"Step V: Symmetry of nonnegative minimizers without plateaus".\\
Let $U=(u_1,\ldots,u_m)\in W^{1,p}_+\cap S_c$ be a minimizer of \eqref{Mc}. 
If, for some $i\in\{1,\ldots,m\}$, $J_i(s,\cdot)$ is strictly convex for every $s\in \RR^+$, then
$\norm{Du_i}_{L^p}=\norm{Du_i^*}_{L^p}$ by Theorem~\ref{thm:minsym}. If, in addition,
$\nabs{C^*_i}=0$, then Corollary~\ref{cor:minsym} implies that $u_i=u_i^{*}$ up to a translation.

\bibliographystyle{plain}
\bibliography{wsconv-symm}

\begin{thebibliography}{10}

\bibitem{Ba89a}
J.~M. Ball.
\newblock A version of the fundamental theorem for {Young} measures.
\newblock In M.~Rascle, D.~Serre, and M.~Slemrod, editors, {\em PDEs and
  continuum models of phase transitions. Proceedings of an NSF-CNRS joint
  seminar held in Nice, France, January 18-22, 1988}, volume 344 of {\em
  Lect.~Notes Phys.}, pages 207--215, Berlin etc., 1989. Springer.

\bibitem{BroSo00}
F.~Brock and A.Y. Solynin.
\newblock An approach to symmetrization via polarization.
\newblock {\em Trans. Am. Math. Soc.}, 352(4):1759--1796, 2000.

\bibitem{BroZie88a}
J.~Brothers and W.~Ziemer.
\newblock Minimal rearrangements of {Sobolev} functions.
\newblock {\em J.~Reine Angew.~Math.}, 384:153--179, 1988.

\bibitem{BuHa06a}
A.~Burchard and H.~Hajaiej.
\newblock Rearrangement inequalities for functionals with monotone integrands.
\newblock {\em J. Funct. Anal.}, 233(2):561--582, 2006.

\bibitem{Carlier}
G.~Carlier.
\newblock On a class of multidimensional optimal trasportation problems.
\newblock {\em Journal of Convex Analysis}.

\bibitem{DaMu98a}
Gianni Dal~Maso and Fran\c{c}ois Murat.
\newblock Almost everywhere convergence of gradients of solutions to nonlinear
  elliptic systems.
\newblock {\em Nonlinear Anal., Theory Methods Appl.}, 31(3-4):405--412, 1998.

\bibitem{EvGa87a}
L.C. Evans and R.F. Gariepy.
\newblock {Some remarks concerning quasiconvexity and strong convergence.}
\newblock {\em Proc. R. Soc. Edinb., Sect. A}, 106:53--61, 1987.

\bibitem{HS1b}
H.~Hajaiej.
\newblock Generalized {Polya-Szego} inequality.
\newblock Preprint, arXiv:1007.0176v1.

\bibitem{H3}
H.~Hajaiej.
\newblock On the necessity of the assumptions used to prove
  {H}ardy-{L}ittlewood and {R}iez rearrangement inequalities.
\newblock Preprint, arXiv:1003.3166v1.

\bibitem{H4}
H.~Hajaiej.
\newblock Symmetric ground state solutions for an m-coupled nonlinear
  schrodinger equation.
\newblock {\em Nonlinear Analysis: Methods, Theory and Applications}, 71(2),
  2009.

\bibitem{H1}
H . Hajaiej.
\newblock Explicit approximation to symmetrization via iterated polarizations.
\newblock {\em Journal of Convex Analysis}, 17, 2010.

\bibitem{HS1a}
H.~Hajaiej and M.~Squassina.
\newblock Generalized {Polya-Szego} inequality and applications to some
  quasi-linear elliptic problems.
\newblock Preprint, arXiv:0903.3975v7.

\bibitem{HS2}
H.~Hajaiej and M.~Squassina.
\newblock Symmetry of minimizers in quasilinear constrainted problems.
\newblock Preprint, arXiv:1004.3384v1 and arXiv:1003.1389v2.

\bibitem{Io77a}
A.D. Ioffe.
\newblock On lower semicontinuity of integral functionals. {I,II}.
\newblock {\em SIAM J. Control Optimization}, 15:521--538, 991--1000, 1977.

\bibitem{Ka85B}
Bernhard Kawohl.
\newblock {\em Rearrangements and convexity of level sets in PDE}, volume 1150
  of {\em Lecture Notes in Mathematics}.
\newblock Springer, Berlin etc., 1985.

\bibitem{Kroe09cp}
Stefan Kr{\"{o}}mer.
\newblock On compactness of minimizing sequences subject to a linear
  differential constraint.
\newblock Preprint 09-CNA-005, submitted.

\bibitem{LiLo01B}
Elliott~H. Lieb and Michael Loss.
\newblock {\em Analysis. 2nd ed.}, volume~14 of {\em Graduate Studies in
  Mathematics}.
\newblock American Mathematical Society (AMS), Providence, RI, 2001.

\bibitem{Mue99a}
Stefan M{\"{u}}ller.
\newblock Variational models for microstructure and phase transisions.
\newblock In S.~Hildebrandt, editor, {\em Calculus of variations and geometric
  evolution problems. Lectures given at the 2nd session of the Centro
  Internazionale Matematico Estivo (CIME), Cetraro, Italy, June 15-22, 1996},
  volume 1713 of {\em Lect.~Notes Math.}, pages 85--210, Berlin, 1999.
  Springer.

\bibitem{VS}
J.~Van Schaftingen.
\newblock Explicit approximation of the symmetric rearrangement by
  polarizations.
\newblock {\em Archiv der Mathematik}.
\newblock To appear.

\bibitem{St1}
C.~A. Stuart.
\newblock Guidance properties for nonlinear planar wave guides.
\newblock {\em Arch. Rat. Mech. Anal}, 125:145--200, 1993.

\bibitem{SuWaWi07a}
J.~Su, Z.-Q. Wang, and M.~Willem.
\newblock Weighted {Sobolev} embedding with unbounded and decaying radial
  potentials.
\newblock {\em J. Differ. Equations}, 238(1):201--219, 2007.

\bibitem{Sy98a}
M.A. Sychev.
\newblock Young measure approach to characterization of behaviour of integral
  functionals on weakly convergent sequences by means of their integrands.
\newblock {\em Ann.~Inst.~Henri Poincar\'{e}, Anal.~Non Lin\'{e}aire},
  15(6):755--782, 1998.

\bibitem{Sy99a}
M.A. Sychev.
\newblock A new approach to young measure theory, relaxation and convergence in
  energy.
\newblock {\em Ann.~Inst.~Henri Poincar\'{e}, Anal.~Non Lin\'{e}aire},
  16(6):773--812, 1999.

\bibitem{Vi84a}
A.~Visintin.
\newblock Strong convergence results related to strict convexity.
\newblock {\em Commun. Partial Differ. Equations}, 9:439--466, 1984.

\bibitem{Zha89a}
Kewei Zhang.
\newblock A weak-strong convergence theorem and its applications.
\newblock {\em Northeast. Math. J.}, 5(1):11--26, 1989.

\end{thebibliography}
\end{document}